\documentclass[reqno]{amsart}
\usepackage{amsmath,amssymb,amsthm}
\usepackage{color,comment,soul}
\newcommand{\inter}{\mathrm{int}\,}
\newcommand{\extr}{\mathrm{extr}\,}
\newcommand{\R}{\mathbb{R}}
\newcommand{\N}{\mathbb{N}}
\newcommand{\SC}{\inter \R^n_+}

\newtheorem{theorem}{Theorem}
\newtheorem{lemma}{Lemma}
\newtheorem{proposition}{Proposition}
\theoremstyle{definition}
\newtheorem{example}{Example} 

\title[Upper and Lower Bounds]{Upper and Lower Bounds for the Iterates of Order-Preserving Homogeneous Maps on Cones}
\author[P. Chodrow]{Philip Chodrow}
\address{Philip Chodrow, Swarthmore College}
\email{pchodro1@swarthmore.edu}

\author[C. Franks]{Cole Franks}
\address{Cole Franks, University of South Carolina}
\email{franksw@email.sc.edu}

\author[B. Lins]{Brian Lins$^*$}
\address{Brian Lins, Hampden-Sydney College}
\email{blins@hsc.edu}
\thanks{$^*$Corresponding author}

\thanks{This work was partially supported by NSF grant DMS-0751964}

\begin{document}
\begin{abstract}
We define upper bound and lower bounds for order-preserving homogeneous of degree one maps on a proper closed cone in $\R^n$ in terms of the cone spectral radius.  We also define weak upper and lower bounds for these maps.  For a proper closed cone $C \subset \R^n$, we prove that any order-preserving homogeneous of degree one map $f: \inter C \rightarrow \inter C$ has a lower bound.  If $C$ is polyhedral, we prove that the map $f$ has a weak upper bound.  We give examples of weak upper bounds for certain order-preserving homogeneous of degree one maps defined on the interior of $\R^n_+$. 
\end{abstract}

\maketitle 

\section{Introduction}

A \textit{closed cone} $C \subset \R^n$ is a closed convex set such that (i) $C \cap (-C) = \{0 \}$ and (ii) $\lambda C = C$ for all $ \lambda \geq 0$. If $C$ has nonempty interior, we say that $C$ is a \textit{proper closed cone}.  For any proper closed cone, the \textit{dual cone} $C^* = \{x \in \R^n : \langle x,y\rangle \geq 0 \, \forall y \in C \}$ is a proper closed cone.  

Any closed cone $C$ defines a partial ordering on $\R^n$ by $x \leq_C y$ if and only if $y-x \in C$.  When the cone $C$ is understood, we will write $\leq$ instead of $\leq_C$.  Let $D \subset \R^n$ be a domain in $\R^n$.  A map $f:D \rightarrow \R^n$ is said to be \textit{order-preserving} if and only if $f(x) \leq f(y)$ whenever $x \leq y$.  It is called \textit{order-reversing} if $f(x) \geq f(y)$ whenever $x \leq y$.  We say that $f:D \rightarrow \R^n$ is homogeneous of degree $\alpha$ if $f(\lambda x) = \lambda^\alpha f(x)$ for all $\lambda > 0$ and $x \in D$.  Order-preserving homogeneous of degree one maps from a cone into itself have been extensively studied (see e.g., \cite{LemmensNussbaum}).  They are a natural extension of the nonnegative matrices, and there are many examples of such maps in applications \cite{Nu89}.  Many important properties of nonnegative matrices generalize to order-preserving homogeneous of degree one maps on cones.  

Let $C$ be a proper closed cone and let $f: C \rightarrow C$ be a continuous order-preserving homogeneous of degree one map.  We define the \textit{cone spectral radius} of $f$ to be 
\begin{equation} \label{spectralRadiusEQ}
\rho_C = \rho_C(f) = \lim_{k \rightarrow \infty} ||f^k(x)||^{1/k},
\end{equation}
for any $x \in \inter C$. The value of $\rho_C$ is independent of $x$ \cite[Proposition 5.3.6]{LemmensNussbaum}. Once again, when the cone $C$ is understood, we will write $\rho$ instead of $\rho_C$.  The well known Krein-Rutman theorem \cite[Corollary 5.4.2]{LemmensNussbaum} asserts that any continuous order-preserving homogeneous of degree one map $f:C \rightarrow C$ has an eigenvector $x \in C$ such that $f(x) = \rho x$.  Note that any such eigenvector will serve as a lower bound on the iterates of $f$ in the following sense.  If $x \leq y$, then $f^k(x) \leq f^k(y)$ for all $k \in \N$ because $f$ is order-preserving.  Thus $f^k(y) \geq \rho^k x$ for all $k$.  

In applications, we often cannot say that the map is defined continuously on the entire closed cone.  There are several examples of important order-preserving homogeneous of degree one maps that are defined only on the interior of the cone.  
Suppose now that $f:\inter C \rightarrow \inter C$ is order-preserving homogeneous of degree one map. The spectral radius of $f$ given in (\ref{spectralRadiusEQ}) is still well defined \cite[Theorem 2.2]{MaNu02}.  Since $f(\inter C) \subseteq \inter C$, it follows that $\rho(f) > 0$.  This is because there exists $\alpha > 0$ such that $f(x) \geq \alpha x$, so $f^k(x) \geq \alpha^k x$.  Since $C$ is a closed cone in a finite dimensional space, it is normal \cite[Lemma 1.2.5]{LemmensNussbaum}, so there exists a $c > 0$ such that $y \geq x$ implies $||y|| \geq c||x||$.  In particular $||f^k(x)|| \geq c||\alpha^k x||$, so $\rho(f) \geq \alpha > 0$.   

We say that $y \in C \backslash \{ 0 \}$ is a \textit{lower bound} for $f$ if for all $x \in \inter C$, $x \geq y$ implies that $f(x) \geq \rho y$.  Note that any eigenvector of $f$ is a lower bound.  Unfortunately, the Krein-Rutman theorem does not apply in general if $f$ is only defined on the interior of the cone.  Without an eigenvector corresponding to the spectral radius, it is not clear that a lower bound must exist.  We address this question in the next section, but for now, let us introduce a weaker notion.  We say that $w \in C^*$ is a \textit{weak lower bound} for $f$ if there exists $x \in \inter C$ such that $\langle f^k(x), w \rangle \geq \rho^k \langle x,w \rangle$ for all $k \geq 0$. 

For $f:\inter C \rightarrow \inter C$ order-preserving and homogeneous of degree one we say that $y \in \inter C$ is an \textit{upper bound} if $x \leq y$ implies that $f(x) \leq \rho y$.  Unlike lower bounds, we do not allow upper bounds on the boundary of $C$.  After all, if $y$ is contained in the boundary of $C$, then there is no $x \in \inter C$ such that $x \leq y$ so it does not make sense to refer to $y$ as an upper bound. We say that $w \in C^*$ is a \textit{weak upper bound} for $f$ if there exists $x \in \inter C$ such that $\langle f^k(x),w \rangle \leq \rho^k \langle x,w \rangle$.  A weak upper or lower bound is uniform in the following sense.  
\begin{lemma}
Let $C$ be a proper closed cone in $\R^n$ and $f: \inter C \rightarrow \inter C$ be order-preserving and homogeneous of degree one.  If $w$ is a weak lower (upper) bound for $f$ and $y \in \inter C$, then there exists a constant $c > 0$ such that $\langle f^k(y),w \rangle \geq c \rho^k$ (respectively,  $\langle f^k(y),w \rangle \leq c \rho^k$).  
\end{lemma}
\begin{proof}
If $w$ is a weak lower bound for $f$, then there exists $x \in \inter C$ such that $\langle f^k(x),w \rangle \geq \rho^k \langle x, w \rangle$.  Since both $x, y \in \inter C$, there exists $\alpha > 0$ such that $\alpha x \le y $.  Applying the map $f^k$,
$$\alpha f^k(x) \le f^k(y)$$
$$\alpha \langle f^k(x),w \rangle \le \langle f^k(y),w \rangle $$
Therefore 
$$\alpha \langle x, w \rangle \le \alpha \langle f^k(x),w \rangle \le \langle f^k(y),w \rangle \le \beta \langle f^k(x),w \rangle$$
Letting $c = \alpha \langle x,w \rangle$ completes the proof if $w$ is a weak lower bound.  The proof for weak upper bounds is essentially the same.   
\end{proof}

As the following lemma shows, the notion of a weak upper or lower bound is indeed weaker than an upper or lower bound.  
\begin{lemma}
Let $C$ be a proper closed cone in $\R^n$ and $f: \inter C \rightarrow \inter C$ be order-preserving and homogeneous of degree one.  If $f$ has a lower (upper) bound, then there is a weak lower (weak upper) bound on the iterates of $f$.   
\end{lemma}
\begin{proof}
Suppose that $y$ is a lower bound for $f$.  Then for every $x \in \inter C$, there is a maximal $\lambda > 0$ such that $x \geq \lambda y$ and it is clear that $\lambda y$ is also a lower bound.  Thus $x - \lambda y \in \partial C$ where $\partial C$ denotes the boundary of $C$.  We may choose $w \in C^* \backslash \{0 \}$ such that $\langle x - \lambda y,w \rangle = 0$ and $\langle x,w \rangle = \langle \lambda y,w \rangle$.  Since $\lambda y$ is a lower bound, $f^k(x) \geq \rho^k \lambda y$ for all $k \in \N$.  Therefore $\langle f^k(x), w \rangle \geq \rho^k \langle \lambda y,w \rangle = \rho^k \langle x,w \rangle$.   The proof for upper bounds is essentially the same. 
\end{proof}

We will prove that for any proper closed cone $C$ and any order-preserving homogeneous of degree one map $f: \inter C \rightarrow \inter C$, the map $f$ has a lower bound.  For order-preserving homogeneous of degree one map on the standard cone $\R^n_+ = \{ x \in \R^n : x_i \geq 0 \text{ for all } i \in \{1,...,n\} \}$, we show in section \ref{section3} that there is a formal eigenvector that is almost an upper bound for the map.  In particular this will establish a weak upper bound for the iterates of any such map on the standard cone.  We then extend this result to show that on any polyhedral cone there is always a weak upper bound for the iterates of any order-preserving homogeneous of degree one map defined on the interior.

\section{Lower Bounds} \label{section2}

Let $C$ be a proper cone in $\R^n$ and let $f:\inter C \rightarrow \inter C$ be order-preserving and homogeneous of degree one.  It is known \cite[Theorem 2.10]{BuNuSp03} that if $C$ is a polyhedral cone, then $f$ has a continuous extension to $C$ that is order-preserving and homogeneous of degree one.  By the Krein-Rutman theorem this extension has an eigenvector $y \in C \backslash \{ 0 \}$ with eigenvalue equal to the cone spectral radius $\rho(f)$.  This proves that order-preserving, homogeneous of degree one self-maps of the interior of a closed polyhedral cone must have a lower bound.  

When the cone $C$ is not polyhedral, the map $f$ might not extend continuously to the boundary of $C$.  In this case, however, there must still be a lower bound.

\begin{theorem} \label{lowerbound}
Let $C$ be any proper closed cone in $\R^n$ and suppose $f: \inter C \rightarrow \inter C$ is order-preserving and homogeneous of degree one with cone spectral radius $\rho(f) = \rho$.  There exists $y \in C \backslash \{ 0 \}$ such that for any $x \in \inter C$ with $x \geq y$, $f^k(x) \geq \rho^k y$ for all $k \in \N$. 
\end{theorem}

\begin{proof}
Fix $v \in \inter C^*$ and $x_0 \in \inter C$ and let $f_\epsilon(x) = f(x) + \epsilon \langle x,v \rangle x_0$ for $\epsilon >0$.  By \cite[Theorem 5.4.1]{LemmensNussbaum}, each map $f_\epsilon$ has an eigenvector $y_\epsilon \in \inter C$ with $||y_\epsilon||=1$ and $f_\epsilon(y_\epsilon) = \rho_\epsilon y_\epsilon$ where $\rho_\epsilon$ is the cone spectral radius of $f_\epsilon$. Furthermore, $\lim_{\epsilon \rightarrow 0} \rho_\epsilon = r$ exists.  Since $\rho_\epsilon > \rho$ for every $\epsilon$, it follows that $r \geq \rho$.  We may choose a sequence $\{\epsilon_i\}_{i \in \N}$ such that $\epsilon_i \rightarrow 0$ and $y_{\epsilon_i} \rightarrow y$ where $y \in C$.  

For a fixed $x \in \inter C$ with $x \ge y$, and for each $y_\epsilon$, there exists a maximal  $\lambda_\epsilon>0$ such that $x \geq \lambda_\epsilon y_\epsilon$.  In particular, $x - \lambda_\epsilon y_\epsilon \in \partial C$, where $\partial C$ denotes the boundary of $C$.  We claim that $\inf_{\epsilon > 0} \lambda_\epsilon > 0$.  Since $x \in \inter C$, there exists $\delta > 0$ such that $x - \delta z \in \inter C$ for all $z \in \R^n$ with $||z|| =1$.  Thus $x - \delta y_\epsilon \in \inter C$ for all $\epsilon > 0$.  So $x \geq \delta y_\epsilon$ for all $\epsilon$ and therefore $\inf \lambda_\epsilon \geq \delta$.  By taking a refinement if necessary, we may assume that $\lambda_{\epsilon_i} \rightarrow \lambda$ where $\lambda \geq \inf_{\epsilon>0} \lambda_\epsilon > 0$.  Then $x-\lambda_{\epsilon_i} y_{\epsilon_i}$ converges to $x - \lambda y$ and since $\partial C$ is a closed set, $x - \lambda y \in \partial C$.  Given that $x \geq y$, $\lambda \geq 1$. 

Since $x \geq \lambda_\epsilon y_\epsilon$ for each $\epsilon$, $f^k_\epsilon(x) \geq \lambda_\epsilon f^k_\epsilon(y_\epsilon) = \lambda_\epsilon \rho^k_\epsilon y_\epsilon$ for any $k \in \N$.  Taking a limit, we see that $f^k(x) \geq \lambda r^k y \geq \rho^k y$ for all $k \in \N$.  
%By the continuity of the inner product, $\langle y_{\epsilon_i},w \rangle \rightarrow \langle y,w\rangle$.  Since $x \geq \lambda_\epsilon y_\epsilon$, for any $k \in \N$, $f_\epsilon^k(x) \geq \lambda_\epsilon \rho_\epsilon^k y_\epsilon$.  Thus $\langle f_{\epsilon_i}^k(x), w \rangle \geq \lambda_{\epsilon_i} \rho_{\epsilon_i}^k \langle y_{\epsilon_i}, w \rangle$.  
%Taking the limit in the expression above, we see that $\langle f^k(x), w \rangle \geq \lambda \rho^k \langle y,w \rangle = \rho^k \langle x,w \rangle.$   
\end{proof}

It has been noted that order-preserving homogeneous of degree one maps on the interior of symmetric cones have weak lower bounds \cite[Corollary 21]{GaubertVigeral}. The result above implies that both weak lower bounds and lower bounds will exist for order-preserving homogeneous of degree one self-maps of the interior of any proper closed cone.   

\section{Upper Bounds on the Standard Cone} \label{section3}

Even in the standard cone an order-preserving homogeneous of degree one map $f: \inter \R^n_+ \rightarrow \inter \R^n_+$ might not have an upper bound.  For example, the matrix $A = \begin{bmatrix} 1 & 1 \\ 0 & 1 \end{bmatrix}$ defines a linear transformation on the standard cone $\R^2_+$ with spectral radius 1, but the iterates of any vector with positive entries under powers of $A$ forms an unbounded sequence.  The following theorem shows that if we relax the definition of an upper bound slightly, then we get a kind of upper bound that is not in $C$ but has all of the other properties of an upper bound for the map $f$.  

\begin{theorem} \label{formaleig}
Let $f: \inter \R^n_+ \rightarrow \inter \R^n_+$ be order-preserving, homogeneous of degree one.  Let $\rho=\rho(f)$ be the cone spectral radius of $f$.  The map $f$ extends continuously to an order-preserving map on $(0,\infty]^n$ and there exists $z \in (0,\infty]^n$ such that $z$ has at least one finite entry and $f(z) = \tilde{\rho} z$ where $\tilde{\rho} \leq \rho$.  In particular, if $x \leq z$, then $f(x) \leq \rho z$.  
\end{theorem}
We refer to $z$ in the theorem above as a \textit{formal eigenvector} of $f$. A formal eigenvector satisfies all of the properties of an upper bound for $f$, except that it is not an element of $\R^n_+$.  Before we prove Theorem \ref{formaleig}, we need some definitions and a lemma.  

In what follows, let $L:[0,\infty]^n$ be the entry-wise reciprocal map:
\begin{equation} \label{L}
(Lx)_i = 
\begin{cases} 
	x_i^{-1} & \text{if~}x_i \in (0,\infty) \\
	x_i = \infty & \text{if~}x_i = 0 \\
	x_i = 0 & \text{if~}x_i = \infty 
\end{cases}
\end{equation}
The set $[0,\infty]^n$ has the obvious partial order $x \le y$ if $x_i \le y_i$ for all $i \in \{1,\ldots n\}$.  Note that $L$ is order-reversing with respect to this partial ordering.  Furthermore, $L(\lambda x) = \lambda^{-1} x$ for all $\lambda \in (0,\infty)$ so $L$ is homogeneous of degree $-1$ on $\inter \R^n_+$.    
\begin{lemma} \label{r(LTL)}
Let $f: \SC \rightarrow \SC$ be an order-preserving homogeneous of degree one map and let $L$ be given by (\ref{L}).  Then $\rho(L \circ f \circ L)^{-1} \leq \rho(f)$.  
\end{lemma}
\begin{proof}
For any $x \in \inter \R^n_+$ and $k \in \N$, $(L \circ f \circ L)^k(x) = L \circ f^k \circ L(x)$.  By the Cauchy-Schwarz inequality, $||Lz|| ||z|| \geq n$ for all $z \in \inter \R^n_+$.  Thus $|| L \circ f^k \circ L(x)||^{1/k} ||f^k(L(z))||^{1/k} \geq n^{1/k}$.  Taking the limit as $k \rightarrow \infty$ we see that $\rho(f) \ge \rho(L \circ f \circ L)^{-1}$.  
\end{proof}

\begin{proof}[Proof of Theorem \ref{formaleig}.] Let $L: \SC \rightarrow \SC$ be given by (\ref{L}).  The map $L \circ f \circ L$ is an order-preserving homogeneous of degree one map on $\SC$.  By \cite[Theorem 2.10]{BuNuSp03}, $L \circ f \circ L$ extends continuously to an order-preserving map on the entire cone $\R^n_+$.  By the Krein-Rutman theorem, the continuous extension of $L\circ f \circ L$ must have an eigenvector $y \in \R^n_+\backslash \{ 0 \}$ with eigenvalue $\tilde{\rho}^{-1}$ where $\tilde{\rho} = \rho(L \circ f \circ L)^{-1}$.  By Lemma \ref{r(LTL)}, $\tilde{\rho} \leq \rho = \rho(f)$.  Since $L$ is a continuous order-reversing bijection from $\R^n_+$ onto $(0,\infty]^n$, it follows that $f$ extends continuously to an order-preserving map on $(0,\infty]^n$.  Let $z = L(y)$.  Since $y \neq 0$, at least one entry of $z$ is finite and $f(z) = L \circ L \circ f(L(y)) = L(L \circ f \circ L (y)) = L(\tilde{\rho}^{-1}(y)) = \tilde{\rho} z$.   
%Choose any $x \in \SC$ such that $x \leq z$.  Then $Lx \ge y$, so 
%$$LTL(Lx) \ge LTLy = \tilde{r}y$$
%$$LT(x) \ge \tilde{r} y$$
%$$Tx \le L(\tilde{r}y) = \tilde{r}^{-1} Ly = \tilde{r}^{-1} z$$  
%Since $\tilde{r}^{-1} \leq r(T)$ by Lemma \ref{r(LTL)}, $Tx \leq r(T) z$ for all $x \leq z$.  
\end{proof}

The continuous extension of an order-preserving homogeneous of degree one map to the set $(0,\infty]^n$ is the crucial insight in the proof above.  This extension is noted in the context of order-preserving additively homogeneous maps in \cite{AkianGaubertGuterman11}.  

\begin{example}
Let $J_n(\lambda)$ denote the $n$-by-$n$ Jordan matrix with eigenvalue $\lambda > 0$, 
$$J_n(\lambda) = \begin{bmatrix} \lambda & 1 &  &  0 \\  
 & \ddots & \ddots &  &    \\
  & & \ddots & 1 \\
 0  & & & \lambda 
 \end{bmatrix}$$
The linear transformation corresponding to $J_n(\lambda)$ maps the interior of $\R^n_+$ into itself.  Up to scalar multiplication, the unique Perron eigenvector is $e_1 = [1, 0, \ldots,0]^T$, and the spectral radius of $J_n(\lambda)$ is $\lambda$.  The formal eigenvector corresponding to $J_n(\lambda)$ (which is unique up to scalar multiplication) is $[\infty, \ldots, \infty, 1]$.  Thus, Theorem \ref{formaleig} tells us that $(J_n(\lambda)^k(x))_n \leq \lambda^k x_n$ for all $k \in \N$.   
\end{example}

\begin{example}
The ``DAD maps" were introduced in \cite{Menon} to solve the problem of trying to find diagonal matrices $D_1$ and $D_2$ for a nonnegative $m$-by-$n$  matrix $A$ such that $D_1 A D_2$ is doubly stochastic.  Such matrices exist if and only if the nonlinear map $f:x \mapsto LA^TLAx$ has an eigenvector with all positive entries. This in turn occurs if and only if the matrix $A$ is a direct sum of fully indecomposible matrices \cite{BrualdiParterSchneider}.   

Note that the map $f$ is order-preserving and homogenous of degree one and it is defined on the interior of $\R^n_+$.  Therefore it must extend continuously to the boundary of $\R^n_+$, and this continuous extension can be computed using the convention that $(0)^{-1} = \infty$ and $(\infty)^{-1} = 0$.  

Assume that $A$ is an $m$-by-$n$ nonnegative matrix such that every row and column contains at least one non-zero entry.  By permuting the rows and columns of $A$, we may assume \cite{Katirtzoglou} that $A$ has the following form, 
$$ A = \begin{bmatrix} A_1 & B_{12} & \ldots & B_{1 \sigma} \\ & A_2 & \ldots & B_{2 \sigma} \\ & 0 & \ddots & \vdots \\ & & & A_\sigma \end{bmatrix}$$
where each $A_i$ is an $m_i$-by-$n_i$ matrix with a corresponding DAD-map $f_i(x) = LA_i^TLA_ix$ with a unique eigenvector $v_i$ (up to scaling) with all positive entries in $\R^{n_i}$ and eigenvalue $\lambda_i = n_i/m_i$.  Furthermore $\lambda_1 \geq \lambda_2 \geq \ldots \geq \lambda_\sigma$.  If $A$ is indecomposable (that is, there are no permutation matrices $P$ and $Q$ such that $PAQ$ is a direct sum of two matrices), then the form above is unique.  Note that the eigenvector $v_\sigma$ can be extended to a formal eigenvector for the map $f(x) = LA^TLAx$ by letting 
$$v = \begin{bmatrix} \infty \\ \vdots \\ \infty \\ v_\sigma \end{bmatrix}.$$
Then $f(v) = LA^TLAv = \lambda_\sigma v$.  Note that the cone spectral radius of $f$ is $\lambda_1$ \cite{Katirtzoglou} so the eigenvalue corresponding to the formal eigenvector may be strictly less than the cone spectral radius.  
\end{example}

\section{Weak Upper Bounds}
\begin{theorem} \label{polybound}
Let $f:\inter C \rightarrow \inter C$ be order-preserving and homogeneous of degree one with $\rho(f) = \rho$.  If $C$ is a polyhedral cone and $x \in \inter C$, then there is a $w \in C^*\backslash \{0 \}$ such that $\langle f^k(x), w \rangle \leq \rho^k \langle x, w \rangle$ for all $k \in \N$.  
\end{theorem}

\begin{proof}
Fix $x_0 \in \inter C$, and a linear functional $v \in \inter C^*$.  Let $f_\epsilon (x) =  f(x) + \epsilon \langle x, v\rangle x_0$ for any $x \in \inter C$.  Let $\rho_\epsilon = \rho(f_\epsilon)$.  By \cite[Theorem 5.4.1]{LemmensNussbaum}, each $f_\epsilon$ has an eigenvector $x_\epsilon \in \inter C$ with $f_\epsilon (x_\epsilon) = \rho_\epsilon x_\epsilon$.  Since we are free to scale the eigenvectors $x_\epsilon$, we will require that $x_\epsilon$ be the smallest scalar multiple of $x_\epsilon$ with the property that $x_\epsilon \geq x$.  For each $\epsilon >0$ there is a unit vector $w_\epsilon \in \extr C^*$ such that $\langle x_\epsilon, w_\epsilon \rangle = \langle x, w_\epsilon \rangle$.  Since $C$ is polyhedral, there are only finitely many extreme rays of $C^*$.  We may choose a sequence $\epsilon_i$ such that $\epsilon_i \rightarrow 0$ and for all $i \in \N$, $w_{\epsilon_i} = w$ where $w$ is a unit vector in $\extr C^*$.  Then, $\langle x_{\epsilon_i}, w \rangle = \langle x, w \rangle$ for all $i$.  We see that 
$$\langle f^k (x),w \rangle \leq \langle f^k_{\epsilon_i} (x_{\epsilon_i}), w \rangle = \rho_{\epsilon_i}^k \langle x, w \rangle.$$    Since $C$ is polyhedral, it follows that the cone spectral radius $\rho_\epsilon$ is continuous \cite[Corollary 5.5.4]{LemmensNussbaum}.  This means that $\rho_\epsilon \rightarrow \rho$ as $\epsilon \rightarrow 0$. Therefore, $\langle f^k(x), w \rangle \leq \rho^k \langle x, w \rangle$ for all $k \in \N$.  
%Note that $\epsilon T y \leq T_\epsilon y$ for all $y \in \inter C$.  We can find a constant $c > 0$ such that $q(y)x_0 \leq c Ty$ for all $y \in \inter C$.  Then $T_\epsilon y \leq \epsilon Ty + (1-\epsilon)c Ty = (\epsilon + (1 - \epsilon) c) Ty$.   Thus,
%$$\epsilon \sqrt[k]{T^k x} \leq \sqrt[k]{T_{\epsilon}^k x} \leq (\epsilon + (1-\epsilon)c) \sqrt[k]{T^k x}$$
%Letting $k \rightarrow \infty$, we obtain:
%$$\epsilon r \leq \rho_\epsilon \leq (\epsilon + (1-\epsilon)c)$$
%and as $\epsilon \rightarrow 1$, we see that $\rho_\epsilon \rightarrow r = 1$.  
\end{proof}

A stronger result than Theorem \ref{polybound} can be obtained for linear maps with an invariant closed cone in $\R^n$ by using the adjoint.  
\begin{theorem}
Let $C$ be a proper closed cone in $\R^n$. Let $T:C \rightarrow C$ be a linear map and let $T^*:C^* \rightarrow C^*$ denote the adjoint of $T$.  Let $\rho=\rho(T)$ be the spectral radius of $T$.  Then there is a $w \in C^* \backslash \{ 0 \}$ such that $\langle T(x),w \rangle =\rho \langle x, w \rangle$ for all $x \in C$. In particular, $w$ is a weak upper bound for the iterates of $T$.    
\end{theorem}
\begin{proof}
Note that the spectral radius of the adjoint $\rho(T^*) = \rho(T)$.  Let $w$ be an eigenvector of $T^*$ with eigenvalue $\rho$, as is guaranteed to exist by the Krein-Rutman theorem.  For any $x \in C$, $\langle T(x),w \rangle = \langle x, T^*(w) \rangle = \rho \langle x,w \rangle$.  
\end{proof}

For an arbitrary proper closed cone $C \subset \R^n$, it is not yet known whether any order-preserving homogeneous of degree one map $f: \inter C \rightarrow \inter C$ must have a weak upper bound.  
\bibliography{UpperLowerBib}
\bibliographystyle{plain}

\end{document}